\documentclass[11pt
]{amsart}

\usepackage{hyperref}
\hypersetup{bookmarksdepth=3}
\usepackage[dvips]{graphicx}
\usepackage{geometry}
\usepackage{amsmath,amssymb}
\usepackage{
mathrsfs}
\usepackage{esint}

\usepackage{tensor}

    \normalbaselines                          %

\newtheorem{thm}{Theorem}[section]
\newtheorem{lm}[thm]{Lemma}

\theoremstyle{definition}

\newtheorem*{df*}{Definition}

\theoremstyle{remark}

\newtheorem*{rem*}{Remark}

\numberwithin{equation}{section}

\newcommand{\la}{\lambda}

\newcommand{\al}{\alpha}

\newcommand{\cz}{Calder\'{o}n--Zygmund\ }

\newcommand{\Om}{\Omega}

\newcommand{\QQ}{[w]_{A_2}}

\def\cyr{\fontencoding{OT2}\fontfamily{wncyr}\selectfont}
\DeclareTextFontCommand{\textcyr}{\cyr}

{\end{list}}

\newcounter{vremennyj}

\begin{document}

\title[$A_2$ conjecture for shifts of complexity $0$ and $1$]{A sharp estimate of weighted dyadic shifts of complexity $0$ and $1$}
\author{Alexander Reznikov}
\address{Department of Mathematics, Michigan State University, East
Lansing, MI 48824, USA}
\author{Sergei Treil}
\address{Dept. of Mathematics, Brown University}
\author{Alexander Volberg}
\address{Department of Mathematics, Michigan State University, East
Lansing, MI 48824, USA}
\thanks{The second and the third authors are grateful to NSF for the support}
\date{}

\begin{abstract}
A simple shortcut to proving sharp weighted estimates for the Martingale Transform and for the  Hilbert transform is presented. It is a unified proof for these both transforms.
\end{abstract}

\maketitle

\section{Introduction}
\label{intro}

Let $\sigma := w^{-1}$.

\bigskip

\noindent{\bf Notations.} We call a shift by $n$ generations, or $SH_n$  any sub-bilinear operator of the following form
$$
(SH_n\,f_1, f_2)=\sum_{J\subset I, |J|= 2^{-n}|I|}2^{-\frac{n}2} c_{IJ}|(f_1,h_I)||(f_2,h_J)|\,,
$$
where $|c_{IJ}|\le 1$.
\begin{thm}
\label{interiorshift1}
$$
(SH_1\,f_1, f_2)\le C\,\,\QQ\|f_1\|_{w}\|f_2\|_{\sigma}\,.
$$
\end{thm}
\begin{proof}
Let
$$
Q:=\QQ\,.
$$
We know from \cite{NTV-2w}, \cite{Wit} that
\begin{thm}
\label{BQ}
There exists a function $B_Q$ of $6$ variables $(X,Y,x,y,u,v)$ defined in $\Omega_Q:=\{(X,Y,x,y,u,v)>0: x^2 \le Xv, y^2\le Yu, 1\le uv \le Q\}$ such that
\begin{equation}
\label{B1}
B_Q(X,Y,x,y,u,v)\le C\, Q\, (X+Y)\,,
\end{equation}
and
\begin{equation}
\label{B2}
d^2\,B_Q(X,Y,x,y,u,v)\ge |dx||dy|\,.
\end{equation}
\end{thm}

In particular, one can conclude that having two points $a= (a_1,...,a_6),\, b=(b_1,...,b_6)$ in $\Omega_Q$ connected by  segment $[a,b]$ {\bf lying entirely inside} $\Omega_Q$ one can introduce the parametrization $c(t)=at +b(1-t)$, consider
$$
q(t)= B_Q(c(t))
$$
and claim, using \eqref{B2} that
\begin{equation}
\label{B3}
- q''(t) \ge |a_3-b_3||a_4-b_4|\,.
\end{equation}

We will need the same thing for some other segments $[a,b]$ {\bf not lying entirely inside} $\Omega_Q$ (but with $a,b\in \Omega_Q$).

The problem is of course that $\Omega_Q$ is not convex.

Now let us apply $B_{40Q}$.
 We choose $I$ and put
 $$
 b:= (\langle f_1^2w\rangle_I, \langle f_2^2\sigma\rangle_{I}, \langle f_1\rangle_I, \langle f_2\rangle_{I}, \langle w\rangle_{I}, \langle \sigma\rangle_I)\,,
 $$
 $$
 b_{+}=  (\langle f_1^2w\rangle_{I_+}, \langle f_2^2\sigma\rangle_{I_+}, \langle f_1\rangle_{I_+}, \langle f_2\rangle_{I_+}, \langle w\rangle_{I_+}, \langle \sigma\rangle_{I_+})\,,
 $$
$$
 b_{-}= (\langle f_1^2w\rangle_{I_-}, \langle f_2^2\sigma\rangle_{I_-}, \langle f_1\rangle_{I_-}, \langle f_2\rangle_{I_-}, \langle w\rangle_{I_-}, \langle \sigma\rangle_{I_-})\,,
 $$
 $$
  b_{ij}= (\langle f_1^2w\rangle_{I_{ij}}, \langle f_2^2\sigma\rangle_{I_{ij}}, \langle f_1\rangle_{I_{ij}}, \langle f_2\rangle_{I_{ij}}, \langle w\rangle_{I_{ij}}, \langle \sigma\rangle_{I_{ij}})\,,
  $$
 where $i,j=\pm$.

 We want to estimate from below
 $$
D:= B_{40Q}(b) -\frac14 (\sum_{i,j=\pm} B_{40Q}(b_{ij}) =A+B+C\,,
 $$
 where
 $$
 A:=  B_{40Q}(b) -\frac12(B_{40Q}(b_+ )+ B_{40Q}(b_-))\,,
 $$
 $$
 B:=\frac12 (B_{40Q}(b_+) -\frac12(B_{40Q}(b_{++} )+ B_{40Q}(b_{+-})))\,,
 $$
 $$
 C:=\frac12 (B_{40Q}(b_-) -\frac12(B_{40Q}(b_{-+} )+ B_{40Q}(b_{--})))\,.
 $$
 Let $b=(\cdot, \cdot, x,y, \cdot,\cdot)$, $b_{+}=(\cdot, \cdot, x+\al,y+\la, \cdot, \cdot)$, $b_{-}=(\cdot, \cdot, x-\al,y-\la, \cdot, \cdot)$,

$$
b_{++}=(\cdot, \cdot, x+\al+\beta_1,y+\la+\delta_1, \cdot, \cdot)\,, b_{+-}=(\cdot, \cdot, x+\al-\beta_1,y+\la-\delta_1, \cdot,\cdot)\,,
$$
$$
b_{-+}=(\cdot, \cdot, x-\al+\beta_2,y-\la+\delta_2, \cdot,\cdot)\,,
b_{--}=(\cdot, \cdot, x-\al-\beta_2,y-\la-\delta_2, \cdot,\cdot)\,.
$$
We do not know the signs of $\al,\la, \beta_1,\beta_2, \delta_1,\delta_2$.

We want to show that there exists an absolute positive constant $c$ such that

\begin{equation}
\label{c}
D\ge c\, |\al|(|\delta_1|+|\delta_2|)\,.
\end{equation}

 Consider several cases. First of all notice that not only all $b, b_-, b_+, b_{ij}$ are in $\Om_Q$ but the segments $[b, b_{ij}]$ are in $\Om_{40Q}$. This follows from the following geometric lemma.
\begin{lm}
\label{geomlerez}
Let three point $A, B, C$ be in $\Om_Q$ and let $M=\frac{A+B}{2}$. Assume $[A, B]\subset \Om_Q$ and $[C, M]\subset \Om_Q$. Then $[C,A], [C,B]\subset \Om_{40Q}$.
\end{lm}
\begin{proof}
Let's prove the statement for $[C, A]$. 
\paragraph{Case 1: $C_{1}\leqslant A_{1}$, $C_{2}\leqslant A_{2}$} Then there is nothing to prove, since if we have a line segment with positive slope, who's endpoints are in $\Om_Q$, then the whole segment lies in $\Om_Q$.

\paragraph{Case 2: $C_{1}\geqslant A_{1}$ or $C_{2}\geqslant A_{2}$} 
Without loss of generality, assume $C_{1}\geqslant A_{1}$. 

Denote $S=\frac{A+C}{2}$ --- the middle of $[A, C]$. Denote also $O=[C, M]\cap [B,S]$. Since $[C, M]$ and $[B,S]$ are two medians of the triangle $ABC$, we have that $O$ is the center of $ABC$. Therefore,
\begin{align}
O=\frac{1}{3}B+\frac{2}{3}S,\\
O=\frac{1}{3}C+\frac{2}{3}M.
\end{align}
\begin{center}
\includegraphics[width=1\linewidth]{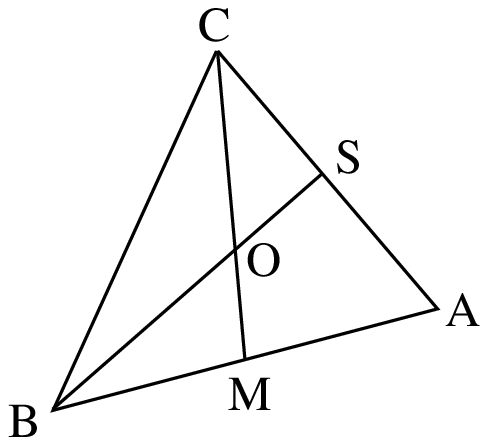}
\end{center}
Therefore, for $k\in \{1,2\}$ we have
\begin{align}
O_{k}\geqslant \frac{2}{3}M_{k},\\
O_{k}\geqslant \frac{1}{3}C_{k}.
\end{align}
On the other hand, 
$$
S_{1}=\frac{A_{1}+C_{1}}{2}\leqslant C_{1}\leqslant 3O_{1}.
$$
Therefore,
$$
S_{1}S_{2}\leqslant 3O_{1}\cdot \frac{3}{2}O_{2}=\frac{9}{2}O_{1}O_{2}.
$$
But $O\in [C,M]\subset \Om_Q$, so
$$
S_{1}S_{2}\leqslant \frac{9}{2}Q.
$$
Therefore, $S\in \Om_{\frac{9}{2}Q}$, and so are $A$ and $C$. Thus, $[A, C]\in \Om_{40Q}$, which finishes the proof.
\end{proof}
The statement for segments $[b, b_{ij}]$ follows from this lemma. Indeed, we have a triangle $bb_{++}b_{+-}$ such that $[b_{++}, b_{+-}]\subset \Om_{Q}$ and, moreover, since endpoints and the middle of the line segment $b_{-}bb_{+}$ are in $\Om_{Q}$, we conclude that $[b,b_{+}]\in \Om_{2Q}$. Therefore, the median of mentioned triangle is in $\Om_{2Q}$, thus, all sides are in $\Om_{40Q}$. 
\begin{center}
\includegraphics[width=1\linewidth]{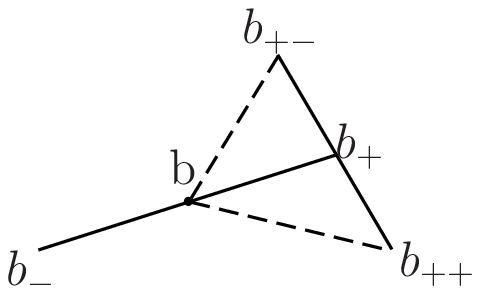}
\end{center}
 \begin{lm}
 \label{geomle}
 Let points $P$, $P_i, i=1,2,3,4$ be in $\Om_Q$ and $P$ be a baricenter of $P_i$. Then all segments $[P,P_i]$ are in $\Om_{40Q}$.
 \end{lm}

 Now fix $i,j$, say, $i=+, j=-$. Consider function
 $$
 f_{+-}(t)= B_{40Q}( t b_{+-} + (1-t) b)
 $$
 and write
 $$
 f_{+-}(0)-f_{+-}(1) =- f'(0) -\frac12 f''(\xi)= -\nabla B_{40Q}(b)\cdot (b_{+-}-b) +\frac12 |x+\al-\beta_1-x||y+\la-\delta_1-y|\,.
 $$
 This is because of Theorem \ref{BQ}.

 We do this for all $f_{ij}, $ $i=\pm, j=\pm$, add and divide by $4$. Then we get the first estimate on $D$:
 $$
 D\ge -\nabla B_{40Q}(b)\cdot (\frac14 (b_{+-}+b_{++}+b_{--}+b_{-+})-b)
 $$
 \begin{equation}
 \label{D1}
 +\frac12 ((|\al-\beta_1||\la-\delta_1|+|\al+\beta_1||\la+\delta_1|) + (|\al-\beta_2||\la-\delta_2|+|\al+\beta_2||\la+\delta_2|))\,.
 \end{equation}

 The first term is zero. If we have the case that $|\beta_1|\le \frac12 |\al|$ and $|\beta_2|\le \frac12 |\al|$, then we get from the first bracket of the second term at least
 $|\al||\delta_1|$, and from the second bracket at least  $|\al||\delta_1|$. In this case \eqref{c} is proved.

 Suppose now that $|\beta_1|\ge \frac12 |\al|$ and $|\beta_2|\ge \frac12 |\al|$. Then we notice that $D=A+B+C$. Moreover, $A\ge0$ as $B_{40Q}$ is concave, and $[b_-, b_+]\subset \Om_{40Q}$
 (see Lemma \ref{geomle}), point $b$ being the center of this segment.
 On the other hand, by Theorem \ref{BQ}
 $$
 2B\ge B_{40Q}(b_+) -\frac12(B_{40Q}(b_{++} )+ B_{40Q}(b_{+-})) \ge c\, |\beta_1||\delta_1|\ge \frac{c}2 |\al | |\delta_1|
 $$
 by our assumption. Symmetrically we will have

 $$
 2c\ge B_{40Q}(b_-) -\frac12(B_{40Q}(b_{-+} )+ B_{40Q}(b_{--})) \ge c\, |\beta_2||\delta_2|\ge \frac{c}2 |\al | |\delta_2|\,.
 $$
 Combining the last two inequalities we also have
 $$
 D= A+B+C \ge c'\,|\al| (|\delta_1|+|\delta_2|)\,,
 $$
 which is \eqref{c} we want.

 Now suppose $|\beta_1|\le \frac12 |\al|$ and $|\beta_2|\ge \frac12 |\al|$.  Then we write $2D= D +A+B+C$. We estimate $D$ by  \eqref{D1}, omitting the the second (positive) bracket of the second term, and writing for the first bracket of the second term the following estimate:
 $$
 D\ge |\al||\delta_1|\,.
 $$
 We again use that $A\ge 0$ and also use that $B\ge 0$ by the same concavity and the fact that $[b_{+-}, b_{++}]\subset \Om_{40Q}$.

 On the other hand, by Theorem \ref{BQ}
 $$
 2C\ge B_{40Q}(b_-) -\frac12(B_{40Q}(b_{-+} )+ B_{40Q}(b_{--})) \ge c\, |\beta_2||\delta_2|\ge \frac{c}2 |\al | |\delta_2|
 $$
 by our assumption $|\beta_2|\ge \frac12 |\al|$. Now combining  $2D= D +A+B+C$ and the last two inequalities we get \eqref{c}.

 We are left with the fourth case: $|\beta_1|\ge \frac12 |\al|$ and $|\beta_2|\le \frac12 |\al|$. But it is totally symmetric to the previous case. So \eqref{c} is always proved.

 \bigskip

 Now we repeat the usual Bellman function summation over dyadic tree (we have above the inequality for the node $I$, we repeat it for nodes $I_+,I_-$ et cetera). In other words we use integration of discrete Laplacian and discrete Green's formula to get (we use also \eqref{B1} of course):

\begin{equation}
\label{sum1}
\frac1{|I|}\sum_{J\subset I}  |L| |\Delta_J f_1|(|\Delta_{J_-} f_2|+\Delta_{J_+}f_2|)\le C\,40\,Q\,(\langle f^2w\rangle_I +\langle g^2\sigma\rangle_{I})\,.
\end{equation}

Our Theorem \ref{interiorshift1}is completely proved.

\end{proof}

\section{Points over i's}
\label{i}

We gave a simple proof of linear estimate of any shift of complexity $1$. So, for example, it gives the way to deduce Stefanie's result from \cite{NTV-2w}. Below we give a very simple proof of \cite{Wit}.   This is up to the existence of $B_Q$. In the next section we give a proof of such an existence.

\section{The heart of the matter: a reduction to bilinear embedding estimate}
\label{bilin}

To prove Theorem \ref{BQ} we need a key inequality. It is an inequality established by Wittwer \cite{Wit} (see also \cite{NTV-2w} on which ]cite{Wit} is based).
\begin{equation}
\label{key}
\sum_{ I} |(\phi w, h_I)||(\psi \sigma,h_I)|\le C\QQ\,\|\phi\|_w\|\psi\|_{\sigma}\,.
\end{equation}

In fact, if \eqref{key} is proved we just put
$$
B_Q (X,Y,x,y,u,v):=\sup\{\frac1{|J|}\sum_{ I\subseteq J} |(\phi w, h_I)||(\psi \sigma,h_I)|: \langle \phi^2 w\rangle_I=X, \langle \psi^2 \sigma\rangle_I=Y,
$$
$$
 \langle \phi \rangle_I=x,  \langle \psi \rangle_I=y,  \langle  w\rangle_I=u,  \langle \psi\rangle_I=v\}\,.
$$

All properties \eqref{B1}--\eqref{B3} can be easily checked as soon as \eqref{key} is proved. We give here an easy proof of \eqref{key}--considerably easier than in \cite{Wit}.
\begin{lm}
\label{hw}
Below $I$'s are dyadic intervals. We have the following decomposition:
$$
h_I = \al_I h_I^w + \beta_I \frac{\chi_I}{\sqrt{I}}\,,
$$
where

1)  $|\alpha_I| \le \sqrt{\langle w\rangle_I}$,

2)$ |\beta_I| \le \frac{|\Delta_I w|}{ \langle \phi\rangle_{I}}$,

3) $\{h_I^w\}_{I} $ is an orthonormal basis in $L^2(w)$,

4) $h_I^w$ assumes on $I$ two constant values, one on $I_+$ and another on $I_-$.
\end{lm}

\vspace{.2in}

We  write

$$
\sum_{ I} |(\phi w, h_I)||(\psi \sigma,h_I)|\le
$$
$$
\sum_{ I}  |(\phi w, h^w_I)|\sqrt{\langle w \rangle_I}|(\psi \sigma,h^{\sigma}_I)|\sqrt{\langle \sigma \rangle_I}\,+
$$
$$
\sum_{I}  |\langle \phi w\rangle_I\frac{|\Delta_I w|}{\langle w \rangle_I}|(\psi \sigma,h^{\sigma}_I)|\sqrt{\langle \sigma \rangle_I}\sqrt{I}\,+
$$
$$
\sum_{I}  |\langle \psi \sigma\rangle_I\frac{|\Delta_I \sigma|}{\langle \sigma \rangle_I}|(\phi w,h^{w}_I)|\sqrt{\langle w \rangle_I}\sqrt{I}\,+
$$
$$
\sum_{I} |\langle \phi w\rangle_I|\langle \psi \sigma\rangle_I \frac{|\Delta_I w|}{\langle w \rangle_I} \frac{|\Delta_I \sigma|}{\langle \sigma \rangle_I}\sqrt{I}\sqrt{I}=: I + II +III +IV\,.
$$

Obviously
\begin{equation}
\label{I}
I\le C\QQ^{1/2} \|\phi\|_w\|\psi\|_{\sigma}\,.
\end{equation}

Terms $II, II$ are symmetric, so consider $II$.
Using Bellman function one can prove now that

\begin{equation}
\label{II}
II\le C\QQ \|\phi\|_w\|\psi\|_{\sigma}\,.
\end{equation}
\begin{equation}
\label{III}
III\le C\QQ \|\phi\|_w\|\psi\|_{\sigma}\,.
\end{equation}

If we do the same in $IV$ by using Cauchy's inequality, we would get
$$
IV\le C\QQ^{3/2}\|\phi\|_{w} \|\psi\|_{\sigma}\,,
$$
which is not our coveted linear estimate. So $I, II, II$ are fine and linear estimate of exterior sum $\overline{\sigma_{11e}}$ is equivalent to the linear estimate of $IV$.

\section{Carleson measures built on $w\in A_2$ and their estimates}
\label{carl}

Let us introduce  {\bf bi-sublinear} sum
$$
B(\phi w, \psi \sigma) :=\sum_{ I}  |\langle \phi w\rangle_I||\langle \psi \sigma\rangle_I| \frac{|\Delta_I w|}{\langle w \rangle_I} \frac{|\Delta_I \sigma|}{\langle \sigma \rangle_I}|I|\,.
$$
Everything is reduced to the estimate of this  {\bf bi-sublinear} sum.

We can rewrite it as
\begin{equation}
\label{Ltrick}
\sum_{I} \frac{|\langle \phi w \rangle_I|}{\langle w \rangle_I} \frac{|\langle \psi \sigma\rangle_I|}{\langle \sigma \rangle_I} |\Delta_I w ||\Delta_I \sigma| |I| \le \QQ\,\|\phi\|_{L^2(L,\sigma)}\|\psi\|_{L^2(L,\sigma)}\,.
\end{equation}

This is immediately reductive to Carleson measure estimate. In fact, the LHS of \eqref{Ltrick} can be rewritten as
\begin{equation}
\label{Carltrick}
\sum_{ I} \frac{|\langle \phi w \rangle_I|}{\langle  w \rangle_I} \frac{|\langle \psi \sigma\rangle_I|}{\langle \sigma \rangle_I} |\Delta_I w ||\Delta_I \sigma| |I| \le B\, \int_L M_w\phi (x) M_{\sigma} \psi  (x) dx\,,
\end{equation}
where $B$ is the Carleson norm of the measure given by the formula
\begin{equation}
\label{kmeas}
\al_I=|\Delta_I w ||\Delta_{I} (\sigma)| |I|\,.
\end{equation}
In fact, \eqref{Carltrick} is a simple geometric argument: exercise!

But the RHS of \eqref{Carltrick}  is estimated by Cauchy inequality independently of $\QQ$ (we learnt this other trick from \cite{CUMP1}):
$$
\int M_w \phi (x)M_{\sigma} \psi  (x) dx = \int M_w \phi (x) M_{\sigma} \psi  (x)\sqrt{w(x)}\sqrt{\sigma(x)} dx\le
$$
$$
 \|M_w \phi\|_w \|M_{\sigma}\chi_L\|_{\sigma} \le A\,\|\phi\|_{L^2(w)}\|\psi\|_{L^2(\sigma)}\,.
$$
Combining this with \eqref{Carltrick} we obtain that everything follows from

\begin{thm}
\label{kCARL}
$$
\|\{\alpha_I\}_I\|_{\text{Carl}} \le  A\, \QQ\,.
$$
\end{thm}
\begin{proof}
In the paper \cite{VaVo1} it is shown that
if for all $I\in D$ we have that two positive functions $u,v$ satisfy
$$
\langle u \rangle _I\langle v\rangle_I \le 1
$$
then
for any $L\in D$ we also have
$$
\frac1{|L|}\sum_{I\in D, I\subset L} |\Delta_I u||\Delta_I v| |I| \le A\, \sqrt{\langle u\rangle_L\langle v\rangle_L}\,.
$$
Take our $w\in A_2$ and put $u= w/\QQ, v=\sigma$. Then the assumption is satisfied, and we immediately get
$$
\sum_{I\in D, I\subset L} |\Delta_I w||\Delta_I \sigma| |I|\le A\, \QQ^{1/2} \sqrt{\langle w\rangle_L\langle \sigma\rangle_L}|L|\,.
$$
In particular, we obtain
$$
\sum_{I\in D, I\subset L} |\Delta_I w||\Delta_I \sigma| |I|\le A\, \QQ\,|L|\,.
$$
This is exactly \eqref{kCARL} for measure $\{\al_I\}_I$!
\end{proof}

\end{document}